\documentclass[11pt,a4paper,english]{amsart}
\usepackage[english]{babel}
\usepackage{amsthm}
\usepackage{amsmath,color}
\usepackage{amssymb}
\usepackage{graphicx}
\usepackage{dsfont}
\usepackage[bf]{caption}
\usepackage{hyperref} 
\usepackage{parskip}

\newtheorem{thm}{Theorem}

\newtheorem{lem}[thm]{Lemma}
\newtheorem{prop}[thm]{Proposition}
\newtheorem{rk}[thm]{Remark}

\def\RR{\mathbb{R}}
\def\ZZ{\mathbb{Z}}

\def\EE{\mathbb{E}}

\newcommand{\p} {\textnormal{\textsf{P}}}
\newcommand {\cro}[1] {\left[ {#1} \right]}
\def \ind {\hbox{ 1\hskip -3pt I}}

\newcommand {\va}[1] {\left| {#1} \right|}

\begin{document}

\title[On the survival probability in the Matheron - De Marsily model]
{On the survival probability in the Matheron - De Marsily model}
\author{Nadine Guillotin-Plantard} 
\address{Institut Camille Jordan, CNRS UMR 5208, Universit\'e de Lyon, Universit\'e Lyon 1, 43, Boulevard du 11 novembre 1918, 69622 Villeurbanne, France.}
\email{nadine.guillotin@univ-lyon1.fr}

\author{Fran\c{c}oise P\`ene}
\address{Universit\'e de Brest and IUF,
LMBA, UMR CNRS 6205, 29238 Brest cedex, France}
\email{francoise.pene@univ-brest.fr}

\keywords{Range; random walk in random scenery; local limit theorem; local time; stable process\\
This research was supported by the french ANR project MEMEMO2}

\begin{abstract}
We are interested in the behaviour of the range and of the first return time to the origin
of random walks in random scenery. As a byproduct a precise estimate of the survival probability 
in the Matheron and de Marsily model \cite{matheron_demarsily} is obtained. Our result confirms the conjectures announced in \cite{Maj, Red2}.

\end{abstract}

\maketitle

\section{Results for random walks in random scenery}

  Random walks in random scenery (RWRS) 
are simple models of processes in disordered
media with long-range correlations. They have been used in a wide
variety of models in physics to study anomalous dispersion in layered
random flows \cite{matheron_demarsily},  diffusion with random sources, 
or spin depolarization in random fields (we refer the reader
to Le Doussal's review paper \cite{ledoussal} for a discussion of these
models). 

On the mathematical side, motivated by the construction of 
new self-similar processes with stationary increments, 
Kesten and Spitzer \cite{KS} and Borodin  \cite{Borodin, Borodin1} 
introduced RWRS in dimension one and proved functional limit theorems. 
This study has been completed in many works, in particular in
\cite{bol} and \cite{DU}.
These processes are defined as follows. Let $\xi:=(\xi_y,y\in \ZZ)$ and $X:=(X_k,k\ge 1)$ 
be two independent sequences of independent
identically distributed random variables taking their values in $\ZZ$. 
The sequence $\xi$ is called the {\it random scenery}. 
The sequence $X$ is the sequence of increments of the {\it random walk}  
$(S_n, n \geq 0)$
defined by $S_0:=0$ and  $S_n:=\sum_{i=1}^{n}X_i$, for $n\ge 1$. 
The {\it random walk in random scenery} (RWRS) $Z$ is 
then defined by
$$Z_0:=0\ \mbox{and}\ \forall n\ge 1,\ Z_n:=\sum_{k=1}^{n}\xi_{S_k}.$$
Denoting by  
$N_n(y)$ the local time of the random walk $S$~:
$$N_n(y):=\#\{k=1,...,n\ :\ S_k=y\} \, ,
$$
it is straightforward to see that 
$Z_n$ can be rewritten as $Z_n=\sum_y\xi_y N_n(y)$.

As in \cite{KS}, the distribution of $\xi_0$ is assumed to belong to the normal 
domain of attraction of a strictly stable distribution 
$\mathcal{S}_{\beta}$ of 
index $\beta\in (0,2]$, with characteristic function $\phi$ given by
$$\phi(u)=e^{-|u|^\beta(A_1+iA_2 \text{sgn}(u))}\quad u\in\mathbb{R},$$
where $0<A_1<\infty$ and $|A_1^{-1}A_2|\le |\tan (\pi\beta/2)|$.
When $\beta > 1$, this implies that $\EE[\xi_0] = 0$.
When $\beta = 1$, we assume
the symmetry condition
$\sup_{ t > 0} \va{\EE\cro{\xi_0 \ind_{\{\va{\xi_0} \le t\}}}} < +\infty \, $.\\
Concerning the random walk, the distribution of $X_1$ is
assumed to belong to the normal basin of attraction of a stable
distribution ${\mathcal S}'_{\alpha}$ with index $\alpha\in (0,2]$, with characteristic function $\psi$ given by
$$\psi(u)=e^{-|u|^\alpha(C_1+iC_2 \text{sgn}(u))}\quad u\in\mathbb{R},$$
where $0<C_1<\infty$ and $|C_1^{-1}C_2|\le |\tan (\pi\alpha/2)|$. In the particular case where $\alpha =1$, we assume that $C_2=0$. Moreover we assume that the additive group $\mathbb Z$ is generated by the support of the distribution of $X_1$.

\noindent Then the following weak convergences hold in the space  of 
c\`adl\`ag real-valued functions 
defined on $[0,\infty)$ endowed with the 
Skorohod $J_1$-topology~:
$$\left(n^{-\frac{1}{\alpha}} S_{\lfloor nt\rfloor}\right)_{t\geq 0}   
\mathop{\Longrightarrow}_{n\rightarrow\infty}
^{\mathcal{L}} \left(Y(t)\right)_{t\geq 0}\, ,$$
$$   \left(n^{-\frac{1}{\beta}} 
\sum_{k=0}^{\lfloor nx\rfloor}\xi_{k}\right)_{x\ge 0}
   \mathop{\Longrightarrow}_{n\rightarrow\infty}^{\mathcal{L}} 
\left(U(x)\right)_{x\ge 0}\quad\mbox{and}\quad
\left(n^{-\frac{1}{\beta}} 
\sum_{k=\lfloor -nx\rfloor}^{1}\xi_{k}\right)_{x\ge 0}
   \mathop{\Longrightarrow}_{n\rightarrow\infty}^{\mathcal{L}} 
\left(U(-x)\right)_{x\ge 0}$$
where $(U(x))_{x\ge 0}$,  $(U(-x))_{x\ge 0}$ and $(Y(t))_{t\ge 0}$ are three independent L\'evy processes such 
that $U(0)=0$, $Y(0)=0$, 
$Y(1)$ has distribution $\mathcal{S}'_{\alpha}$, $U(1)$ and $U(-1)$ 
have distribution  $\mathcal{S}_\beta$.
We will denote by $(L_t(x))_{x\in\mathbb{R},t\geq 0}$ a continuous version with compact support of the local time of the process $(Y(t))_{t\geq 0}$. Let us define
 $$\delta := 1-\frac{1}{\alpha}+ \frac{1}{\alpha \beta}.$$
In the case $\alpha\in (1,2]$ and $\beta\in (0,2]$, Kesten and Spitzer \cite{KS} proved 
the convergence in distribution of $(n^{-\delta} Z_{[nt]})_{t\ge 0}, n\geq 1$ (with respect to the $J_1$-metric), 
to a process $\Delta=(\Delta_t)_{t\geq 0}$ defined in this case by
$$\Delta_t : = \int_{\RR} L_t(x) \, d U(x).$$
This process $\Delta$ is called Kesten-Spitzer process in the literature.

When $\alpha\in(0,1)$ (when the random walk $S$ is transient) and $\beta\in(0,2]\setminus \{1\} $,
$(n^{-\frac 1\beta} Z_{[nt]})_{t\ge 0}, n\geq 1$ converges in distribution (with respect to the $M_1$-metric), 
to $(\Delta_t:=c_0U_t)_{t\geq 0}$ for some $c_0>0$ (see \cite{FFN}).

When $\alpha=1$ and $\beta\in(0,2]\setminus \{1\} $,
$(n^{-\frac 1\beta}(\log n)^{\frac 1\beta-1} Z_{[nt]})_{t\ge 0}, n\geq 1$ converges in distribution (with respect to the $M_1$-metric), 
to $(\Delta_t:=c_1U_t)_{t\geq 0}$ for some $c_1>0$ (see \cite{FFN}).

Hence in any of the cases considered above, $(Z_{\lfloor nt\rfloor}/a_n)_{t\ge 0}$ converges in distribution (with respect to the
$M_1$-metric) to some process $\Delta$, with
$$
a_n:=\left\{ \begin{array}{lll}
n^{1- \frac{1}{\alpha} +\frac{1}{\alpha\beta}} & \text{if} & \alpha\in (1,2]\\
n^{\frac{1}{\beta}} (\log n)^{1-\frac{1}{\beta}} & \text{if} & \alpha =1\\
n^{\frac{1}{\beta}} & \text{if} & \alpha \in (0,1).
\end{array}
\right.
$$
We are interested in the asymptotic behaviour of the range ${\mathcal R}_n$ of the RWRS $Z$, i.e.
of the number of sites visited by $Z$ before time $n$: 
$${\mathcal R}_n:=\#\{Z_0,\dots, Z_{n}\}.$$

\begin{rk}\label{cvpsrange2}
Let $\alpha\in (0,2]$ and $\beta\in (0,1)$. Then the RWRS is transient (see for instance \cite{BFFN1}) and, due to an argument\footnote{We consider the ergodic dynamical system $(\Omega,\mu,T)$ given by $\Omega:=\mathbb Z^{\mathbb Z}\times
\mathbb Z^{\mathbb Z}$, $\mu:=(\mathbb P_{S_1})^{\otimes\mathbb Z}\otimes (\mathbb P_{\xi_1})^{\otimes\mathbb Z}$ and 
$T((\alpha_k)_k,(\epsilon_k)_k):=((\alpha_{k+1})_k,
  (\epsilon_{k+\alpha_0})_k)$ (see for instance \cite{KMC} for its ergodicity, p.162). We set
$f((\alpha_k)_k,(\epsilon_k)_k)=\epsilon_0$.
With these choices, $(Z_j)_{j\geq 1}$ has the same distribution under $\mathbb P$
as $(\sum_{k=1}^j f\circ T^j)_{j\geq 1}$ under $\mu$.
} by Derriennic \cite[Lemma 3.3.27]{Zeitouni}, $({\mathcal R}_n/n)_n$ converges $\mathbb P$-almost surely to $\mathbb P[Z_j\ne 0,\ \forall j\geq 1]$.
\end{rk}

For recurrent random walks in random scenery, we distinguish the 
easiest case when $\xi_1$ takes its values in $\{-1,0,1\}$. In that case, $\beta=2$, $U$ is the standard real Brownian motion,
$$
a_n=\left\{ \begin{array}{lll}
n^{1- \frac{1}{2\alpha}} & \text{if} & \alpha\in (1,2]\\
\sqrt{n \log n} & \text{if} & \alpha =1\\
\sqrt{n} & \text{if} & \alpha \in (0,1)
\end{array}
\right.
$$
and the limiting process $\Delta$ is either the Kesten-Spitzer process (case $\alpha\in (1,2])$ or the real Brownian motion (case $\alpha\in(0,1]$). Remark that in any case the limiting process
is symmetric.

Let $T_0:=\inf\{n\ge 1\ :\ Z_n=0\}$ be the first return time of the RWRS $Z$
to $0$.
\begin{prop}\label{cvpsrange3}
If $\alpha\in(0,2]$ and if $\xi_1$ takes its values
in $\{-1,0,1\}$, then
\begin{equation}\label{EEE1}
\frac{\mathcal R_n}{a_n}=\frac{\sup_{t\in[0,1]}Z_{\lfloor nt\rfloor}-\inf_{t\in[0,1]}Z_{\lfloor nt\rfloor}+1}{a_n} \stackrel{\mathcal L}{\longrightarrow} \sup_{t\in[0,1]}\Delta_t-\inf_{t\in[0,1]}\Delta_t.
\end{equation}
Moreover
\begin{equation}\label{EEE2}
\lim_{n\rightarrow+\infty} \frac{\mathbb{E}[\mathcal R_n]}{a_n}=2\, \mathbb E\left[\sup_{t\in[0,1]}\Delta_t\right]
\end{equation}
and
\begin{equation}\label{EEE3}
\lim_{n\rightarrow +\infty}\frac n{a_n}\mathbb P(T_0>n)=\max\left(2-\frac 1{\alpha}, 1\right) \mathbb E\left[\sup_{t\in[0,1]}\Delta_t\right].
\end{equation}
\end{prop}
The range of RWRS in the general case $\beta\in(1,2]$ is much more
delicate. Indeed, the fact that $\mathcal R_n$  is less than $\sup_{t\in[0,1]}Z_{\lfloor nt\rfloor}-\inf_{s\in[0,1]}Z_{\lfloor ns\rfloor}+1$ will only provide an upper bound; we use a separate argument to obtain the lower bound insuring that $\mathcal R_n$ has order $a_n$.
\begin{prop}\label{cvpsrange4}
Let $\alpha\in(0,2]$ and $\beta\in (1,2]$. Then
\begin{equation}\label{EEE2a}
0<\liminf_{n\rightarrow +\infty}\frac{\mathbb E[\mathcal R_n]}{a_n}\leq \limsup_{n\rightarrow +\infty}\frac{\mathbb E[\mathcal R_n]}{a_n}<\infty
\end{equation}
and
\begin{equation}\label{EEE3a}
0<\liminf_{n\rightarrow +\infty}\frac n{a_n}\mathbb P(T_0>n)\leq \limsup_{n\rightarrow +\infty}\frac n{a_n}\mathbb P(T_0>n)<\infty
\end{equation}
\end{prop}
We actually prove that $\limsup_{n\rightarrow +\infty}\frac{\mathbb E[\mathcal R_n]}{a_n}\le \mathbb E[\sup_{t\in[0,1]}\Delta_t-\inf_{t\in[0,1]}\Delta_t]$.
The question wether $\lim_{n\rightarrow +\infty}\frac{\mathbb E[\mathcal R_n]}{a_n}=\mathbb E[\sup_{t\in[0,1]}\Delta_t-\inf_{t\in[0,1]}\Delta_t]$ or not is still open.\\
\section{Results for a two-dimensional random walk with randomly oriented layers}
We are interested in the survival probability of a particle evolving on a randomly oriented lattice 
introduced by Matheron and de Marsily in \cite{matheron_demarsily} (see also \cite{BGKPS}) to modelise fluid transport in a porous stratified medium. 
Supported by physical arguments, numerical simulations and comparison with the Fractional Brownian Motion, 
Redner \cite{Red2} and Majumdar \cite{Maj} conjectured that the survival probability asymptotically behaves as $n^{-\frac{1}{4}}$. 
In this paper we rigorously prove their conjecture.  
Let us describe more precisely the model and the results. Let us fix $p\in(0,1)$.
The (random) environment will be given by a sequence $\epsilon=(\epsilon_k)_{k\in\mathbb Z}$
of i.i.d. (independent identically distributed) centered random variables with
values in $\{\pm 1\}$ and defined on the probability space
$(\Omega,\mathcal T,\p)$. Given $\epsilon$, the position of the particle $M$ is defined as a
$\mathbb Z^2$-random walk on nearest neighbours starting
from $0$ (i.e. $\mathbb P^\epsilon(M_0=0)=1$) and with transition probabilities
$$\mathbb P^\epsilon(M_{n+1}=(x+ \epsilon_y,y)|M_n=(x,y))=p,\quad
   \mathbb P^\epsilon(M_{n+1}=(x,y\pm 1)|M_n=(x,y))=\frac{1-p}2. $$
At site $(x,y)$, the particle can either get down  (or get up) with probability $\frac{1-p}2$ or move with probability $p$ on the $y'$s horizontal line according to its orientation (to the right (resp. to the left) if $\epsilon_y=+1$ (resp. if $\epsilon_y=-1$)).
We will write $\mathbb P$ for the annealed law, that is the
integration of the quenched distribution $\mathbb P^\epsilon$ with respect to $\p$. 
In the sequel this random walk will be named MdM random walk.
This 2-dimensional random walk in random environment was first rigorously studied by mathematicians in \cite{CP}. They proved that the MdM random walk is transient under the annealed law $\mathbb P$ and under the quenched law $\mathbb P^\epsilon$ for $\p$-almost every environment $\epsilon$. It was also proved that it has speed zero.
Actually the MdM random walk is closely related to RWRS. This fact was first noticed in \cite{GPN}. More precisely its first coordinate can be viewed as a generalized RWRS, the second coordinate being a lazy random walk on $\mathbb Z$ (see Section 5 of \cite{BFFN1} for the details). Using this remark, a functional limit theorem was proved in \cite{GPN} and a local limit theorem was established in \cite{BFFN1}, more precisely there exists some constant $C$ only depending on $p$ such that for $n$ large,
$$\mathbb{P}(M_{2n}=(0,0) ) \sim C n^{-\frac 5 4}.$$
Since the random walk $M$ {\it does not }have the Markov property under the annealed law, we are not able to deduce the survival probability from the previous local limit theorem.
Let us precise that the survival probability is the probability that the particle does not visit the $y-$axis (or the line $x=0$) before time $n$ i.e. $\mathbb{P} (T_0^{(1)} >n)$ where 
$$T_0^{(1)}:=\inf\{n\ge 1\ :\ M^{(1)}_n=0\}$$
is the first return time of the first coordinate $M^{(1)}$ of $M$ to $0$.
As for RWRS the asymptotic behavior of this probability will be deduced from the range $\mathcal R_n^{(1)}$ of the first coordinate  i.e. the number of vertical lines visited by
$(M_k)_k$ up to time $n$, namely
$$\mathcal R_n^{(1)}:=\#\{x\in\mathbb Z\ :\ \exists k=0,...,n,\ \exists y\in\mathbb Z\ :\ M_k=(x,y)\}.$$
Let us recall that in \cite{GPN} the first coordinate $M^{(1)}_{\lfloor nt\rfloor}$ normalized by $n^{\frac 34}$ is shown to converge in
distribution to $K_p\Delta_t^{(0)}$, where $K_p:=\frac p{(1-p)^{\frac 14}}$ and where $\Delta^{(0)}$ is the Kesten-Spitzer process $\Delta$ with $U$ and $Y$ two independent standard Brownian motions.
\begin{prop}[Survival probability of the MdM random walk]\label{cvpsrange3bis}
$(\mathcal R_n^{(1)}/n^{\frac 34})_n$ converges in distribution
to $K_p\left(\sup_{t\in[0,1]}\Delta_t^{(0)}-\inf_{t\in[0,1]}\Delta_t^{(0)}
\right).$
Moreover
\begin{equation}\label{EEE2c}
\lim_{n\rightarrow+\infty} \frac{\mathbb{E}[\mathcal R_n^{(1)}]}{n^{\frac 34}}=2K_p\, \mathbb E\left[\sup_{t\in[0,1]}\Delta^{(0)}_t\right]
\end{equation}
and
\begin{equation}\label{EEE3c}
\lim_{n\rightarrow +\infty}  n^{\frac 1 4}\mathbb P(T_0^{(1)}>n)=\frac 32K_p\, \mathbb E\left[\sup_{t\in[0,1]}\Delta^{(0)}_t\right].
\end{equation}
\end{prop}
\begin{rk} In the historical model \cite{matheron_demarsily}, the probability $p$ is equal to $1/3$, and in this particular case the survival probability is similar to $ \kappa n^{-\frac{1}{4}}$ 
where 
$$\kappa =\left(\frac{3}{2^5}\right)^{1/4} \mathbb E\left[\sup_{t\in[0,1]}\Delta^{(0)}_t\right].$$
An open question is to give an estimation of the above expectation.
\end{rk}
\begin{rk}\label{cvpsrange}
It is worth noticing that the range ${\mathcal R}_n$ of the MdM random walk, i.e. the number of sites visited by $M$ before time $n$: 
${\mathcal R}_n:=\#\{M_0,\dots M_{n}\}$ is well understood.
Using \footnote{
We consider the ergodic dynamical system $(\tilde\Omega,\tilde\mu,\tilde T)$ given by  $\tilde\Omega:=\{-1,1\}^{\mathbb Z}\times
\{-1,0,1\}^{\mathbb Z}$, $\tilde\mu:=(\frac{\delta_1+\delta_{-1}}2)^{\otimes\mathbb Z}\otimes(p\delta_0+\frac{1-p}2\delta_1+\frac{1-p}2\delta_{-1})^{\otimes\mathbb Z}$
and $\tilde T((\epsilon_k)_k,(\omega_k)_k)=((\epsilon_{k+\omega_0})_k,
  (\omega_{k+1})_k)$. We also set $\tilde f((\epsilon_k)_k,(\omega_k)_k)=(\epsilon_0,0)$
if $\omega_0 = 0$, $\tilde f((\epsilon_k)_k,(\omega_k)_k)=(0,\omega_0)$
otherwise. 
We observe that $(M_j)_{j\geq 1}$ has the same distribution under $\mathbb P$
as $(\sum_{k=0}^{j-1} \tilde f\circ \tilde T^j)_{j\geq 1}$ under $\tilde \mu$.
}
again \cite[Lemma 3.3.27]{Zeitouni}, $({\mathcal R}_n/n)_n$ converges $\mathbb P$-almost surely to $\mathbb P[M_j\ne 0,\ \forall j\geq 1]$, which contradicts the result announced in \cite{LeNy}.
\end{rk}
\section{Proofs}\label{Recurrent case}
In this section we prove Propositions  \ref{cvpsrange3}, \ref{cvpsrange4} and  \ref{cvpsrange3bis}.

Observe first that
the asymptotic estimates on the tail distribution function of the first return time to the origin
\eqref{EEE3}, \eqref{EEE3a}, \eqref{EEE3c} 
are direct consequences of respective estimates \eqref{EEE2}, \eqref{EEE2a}, \eqref{EEE2c}
on the mean range.
Indeed 
\begin{eqnarray*}
\mathbb E[\mathcal R_n]&=& 1+\sum_{k=1}^n
        \mathbb P( Z_k \neq Z_{k-1}; ... ; Z_k \neq  Z_{0})\\
      &=& 1+\sum_{k=1}^n \mathbb P( Z_1 \neq 0; ... ; Z_k \neq 0)=1+\sum_{k=1}^n\mathbb P(T_0>k)
\end{eqnarray*} 
by stationarity of the increments of $Z$ under the annealed distribution.
Since $(\mathbb P(T_0>k))_k$ is non increasing,  for every $0<x<1<y$, we have
$$\frac{\mathbb E[\mathcal R_{\lfloor y n\rfloor}-\mathcal R_{n}]}{\lfloor yn\rfloor- n}  \le  \mathbb P(T_0>n)
\le \frac{\mathbb E[\mathcal R_{ n}-\mathcal R_{\lfloor x n\rfloor}]}{n-\lfloor xn\rfloor}.$$
Hence, writing $C_-:=\liminf_{n\rightarrow +\infty}\frac {\mathbb E[\mathcal R_n]}{a_n}$ and $C_+:=\limsup_{n\rightarrow +\infty}\frac {\mathbb E[\mathcal R_n]}{a_n}$, we obtain
$$ \frac {y^\vartheta C_--C_+}{y-1}\le\liminf_{n\rightarrow +\infty}\frac n{a_n}\mathbb P(T_0>n)\le
    \limsup_{n\rightarrow +\infty}\frac n{a_n}\mathbb P(T_0>n)\le
       \frac {C_+-x^\vartheta C_-}{1-x},$$
with $\vartheta:=\max\left(1-\frac 1{2\alpha},\frac 12\right)$.
This will give \eqref{EEE3}, \eqref{EEE3a}; we proceed analogously for \eqref{EEE3c}.
\\
For Propositions  \ref{cvpsrange3},  \ref{cvpsrange3bis},
we observe that 
${\mathcal R}_n= \max_{0\le k\le n}Z_k-\min_{0\le k\le n} Z_k +1$
and ${\mathcal R}_n^{(1)}= \max_{0\le k\le n}M_k^{(1)}-\min_{0\le k\le n} M_k^{(1)} +1$ whereas for 
Proposition \ref{cvpsrange4}, we only have
${\mathcal R}_n\le \max_{0\le k\le n}Z_k-\min_{0\le k\le n} Z_k +1$. Hence the convergence of the means of the range in Propositions \ref{cvpsrange3} and  \ref{cvpsrange3bis} and the upper bound for $\mathbb E[R_n]$ in Proposition \ref{cvpsrange4} will come from lemmas \ref{LEM0} and \ref{LEM0bis} below.

Let us start by the convergence in distribution.

\begin{proof}[Proof of the convergences in distribution]
Due to the convergence for the $M_1$-topology of $((a_n^{-1} Z_{\lfloor nt\rfloor})_t)_n$ to $(\Delta_t)_t$ as $n$ goes to infinity,
we know (see Section 12.3 in \cite{Whitt}) that $(a_n^{-1}(\max_{0\le k\le n}Z_k-\min_{0\le \ell\le n}Z_{\ell}))_n$ converges in distribution
to $\sup_{t\in[0,1]}\Delta_t-\inf_{s\in[0,1]}\Delta_s $
as $n$ goes to infinity.
\\
Due to \cite{GPN}, $((M^{(1)}_{\lfloor nt\rfloor}/n^{\frac 34})_t)_n$ converges in distribution to $(K_p\Delta_t^{(0)})_t$ in the Skorohod space endowed with the $J_1$-metric. 
Hence $(n^{-\frac 34}(\max_{k=0,...,n}M_k^{(1)}-\min_{\ell=0,...,n}M_\ell^{(1)}))_n$ converges in distribution
to $K_p(\sup_{t\in[0,1]}\Delta_t^{(0)}-\inf_{s\in[0,1]}\Delta_s^{(0)})$.
\end{proof} 
\begin{lem}[RWRS]\label{LEM0}
Assume $\beta>1$, then 
$$\lim_{n\rightarrow +\infty}\frac{\mathbb E\left[\max_{k=0,...,n}Z_k\right]}{a_n}=\mathbb E\left[\sup_{t\in[0,1]}\Delta_t \right].$$
\end{lem}
\begin{lem}[First coordinate of the MdM random walk]\label{LEM0bis}
$$\lim_{n\rightarrow +\infty}\frac{\mathbb E\left[\max_{k=0,...,n}M_k^{(1)}\right]}{n^{\frac 34}}=K_p\mathbb E\left[\sup_{t\in[0,1]}\Delta_t^{(0)} \right].$$
\end{lem}
\begin{proof}[Proof of Lemma \ref{LEM0}]
As explained above, we know that $(a_n^{-1}\max_{0\le k\le n}Z_k)_n$ converges in distribution
to $\sup_{t\in[0,1]}\Delta_t $
as $n$ goes to infinity.
Now let us prove that this sequence is
uniformly integrable. 
To this end we will use the fact that, conditionally to
the walk $S$, the increments of $(Z_n)_n$ are centered and positively associated.
Let $\beta'\in(1,\beta)$ be fixed.
Due to Theorem 2.1 of \cite{Gong}, there exists some constant $c_{\beta'}>0$ such that
\begin{eqnarray*}
\mathbb E\left[\left|\max_{j=0,...,n}Z_j\right|^{{\beta'}}|S\right] &\le & \mathbb E\left[\max_{j=0,...,n} |Z_j|^{{\beta'}} |S\right]   \\
& \le& c_{\beta'} \mathbb E\left[|Z_n|^{{\beta'}}|S\right]\\
\end{eqnarray*}
so
\begin{eqnarray*}
\mathbb E\left[\left|\max_{j=0,...,n}Z_j\right|^{{\beta'}}\right] &=& \mathbb E\left[\mathbb E\left[\left|\max_{j=0,...,n}Z_j\right|^{{\beta'}}|S\right]\right]\\
& \le& c_{\beta'} \mathbb E\left[|Z_n|^{{\beta'}}\right].
  \end{eqnarray*}  
It remains now to prove that $\mathbb E[|Z_n|^{\beta'}]=O(a_n^{\beta'})$.
\\
\noindent Let us first consider the easiest case when the random scenery is square integrable that is $\beta=2$, then we take $\beta'=2$ in the above computations and observe that 
$\mathbb E\left[|Z_n|^{2}\right]=\mathbb E[\xi_0^2]\mathbb E[V_n]$, 
where $V_n$ is the number of self-intersections up to time $n$
of the random walk $S$, i.e.  
$V_n=\sum_x(N_n(x))^2=\sum_{i,j=1}^n{\mathbf 1}_{S_i=S_j}$. 
Usual computations (see Lemma 2.3 in \cite{bol}) give that
$$\mathbb E[V_n]=\sum_{i,j=1}^n\mathbb P(S_{i-j}=0) \sim c'(a_n)^2$$
and the result follows.
\\
\noindent When $\beta\in(1,2)$, let us define $V_n(\beta)$ as follows
$$V_n(\beta):=\sum_{y\in\mathbb Z}(N_n(y))^\beta.$$
Given the random walk, $Z_n$ is a sum of independent zero-mean random variables, then
from Theorem 3 in \cite{VBE}, there  exists some constant $C>0$ such that for every $n$
$$\mathbb E[|Z_n|^{\beta'} | S ]\leq C \sum_y N_n(y)^{\beta'} \mathbb E[|\xi_y|^{\beta'}] \leq C V_n(\beta').$$
From which we deduce that $\mathbb E[|Z_n|^{\beta'}] \leq C \mathbb E[ V_n(\beta')]$.

If $\alpha >1$, due to Lemma 3.3 of \cite{NadineClement}, we know that
$\mathbb E[V_n(\beta')]=O\left(a_n^{\beta'}\right)$.
If $\alpha \in (0,1]$, using H\"{o}lder's inequality, we have
$$\mathbb E[V_n(\beta')] \leq \mathbb E[R_n]^{1-\frac{\beta'}{2}} \mathbb E[V_n]^{\frac{\beta'}{2}}.$$
Now if $\alpha =1$, we know that $\mathbb E[R_n] \sim c\frac{n}{\log n}$ (see for instance Theorem 6.9, page 398 in \cite{LGR}) and $\mathbb E[V_n] \sim c n\log n$ so $\mathbb E [V_n(\beta')] = O\left(a_n^{\beta'}\right)$
with $a_n = n^{\frac{1}{\beta'}} (\log n)^{1-\frac{1}{\beta'}}$. In the case $\alpha\in (0,1)$,the random walk is transient and the expectations of $R_n$ and $V_n$ behaves as $n$, we deduce that
 $\mathbb E [V_n(\beta')] = O\left(a_n^{\beta'}\right)$
with $a_n = n^{\frac{1}{\beta'}}$.\\*
We conclude that
$$    \lim_{n\rightarrow+\infty}\mathbb E\left[\max_{j=0,...,n}\frac{Z_j}{a_n}\right]
        = \mathbb E\left[\max_{t\in[0,1]}\Delta_t \right].$$
\end{proof}
\begin{proof}[Proof of Lemma \ref{LEM0bis}]
We know that $(n^{-\frac 34}\max_{k=0,...,n}M_k^{(1)})_n$ converges in distribution
to $K_p\sup_{t\in[0,1]}\Delta_t^{(0)}$.
To conclude, it is enough to prove that this sequence is
uniformly integrable. To this end we will prove that
it is bounded in $L^2$.
\\
Recall that the second coordinate of the MdM random walk is a random walk. Let us write it $(S_n)_n$.
Observe that 
$$M_n^{(1)}:=\sum_{k=1}^n \varepsilon_{S_k}\ind_{\{S_k=S_{k-1}\}}=\sum_{y\in\mathbb Z} \varepsilon_{y}\tilde N_n(y),$$
with $\tilde N_n(y):=\#\{k=1,...,n\ :\ S_k=S_{k-1}=y\}$.
Observe that $\tilde N$ is measurable with respect to the random walk $S$ and that $0\le \tilde N_n(y)\le N_n(y)$.
\\
Conditionally to
the walk $S$, the increments of $(M^{(1)}_n)_n$ are centered and positively associated. It follows from Theorem 2.1 of \cite{Gong} that 
\begin{eqnarray*}
\mathbb E\left[\left|\max_{j=0,...,n}M^{(1)}_j\right|^2|S\right] 
& \le& c_{2} \mathbb E\left[|M^{(1)}_n|^2|S\right]\\
& \le& c_{2} \sum_{y\in\mathbb Z}(\tilde N_n(y))^2 \le c_{2} V_n,
\end{eqnarray*}
where again $V_n=\sum_{y\in\mathbb Z}(N_n(y))^2$.
Therefore
$$
\mathbb E\left[\left|\max_{j=0,...,n}M^{(1)}_j\right|^2\right] \le  c_{2} \mathbb E[V_n].$$
Again the result follows from the fact that $\mathbb E[V_n]\sim c'n^{\frac 32}$.
\end{proof}

\begin{proof}[Proof of the lower bound of Proposition \ref{cvpsrange4}]
Let $\mathcal N_n(x):=\#\{k=1,...,n\,:\,Z_k=x\}$.
Applying the Cauchy-Schwarz inequality to 
$n= \sum_x  \mathcal N_n(x) \mathbf 1_{\{\mathcal N_n(x)>0\}} $, we obtain
$$n^2\le \sum_y \mathbf 1_{\{\mathcal N_n(y)>0\}}\, \sum_x(\mathcal N_n(x))^2 =\mathcal R_n\, \mathcal V_n,$$
with $\mathcal V_n=\sum_{x}(\mathcal N_n(x))^2
         =\sum_{i,j=1}^n\mathbf 1_{\{Z_i=Z_j\}}$
the number of self-intersections of $Z$ up to time $n$
and so using Jensen's inequality,
$$\frac{ \mathbb E[\mathcal R_n]}{a_n} \ge \frac{n^2}{a_n}  \mathbb E[(\mathcal V_n)^{-1}]\ge \frac{n^2}{a_n}{\mathbb E[\mathcal V_n]^{-1}}.$$
Moreover, using the local limit theorems for the RWRS proved in \cite{BFFN1,FFN},
\begin{eqnarray*}
\mathbb E[\mathcal V_n]&=&
    n+ 2 \sum_{1\leq i <j \leq n} \mathbb P(Z_{j-i}=0)  \sim    C' \frac{n^2}{a_n}.   
\end{eqnarray*}
Hence 
    $$\liminf_{n\rightarrow +\infty} \frac{\mathbb E[\mathcal R_n]}{a_n }\ge \frac 1{C'} >0.$$
\end{proof}

\end{document}